\documentclass[11pt,a4paper]{amsart}

\usepackage{amsmath,amssymb}
\usepackage{hyperref}
\usepackage{xcolor}
\usepackage{graphicx}
\usepackage{algorithm}
\usepackage{booktabs}
\let\labelindent\relax
\usepackage{enumitem}
\usepackage{subcaption}
\usepackage[style=ieee,backend=bibtex,url=false,doi=false,isbn=false,date=year,citestyle=numeric-comp,maxbibnames=10,minbibnames=10,maxcitenames=10,mincitenames=10]{biblatex}

\usepackage{xcolor} 

\bibliography{paper}

\newtheorem{thm}{Theorem}
\newtheorem{prop}{Proposition}
\newtheorem{rem}{Remark}
\newtheorem{lem}{Lemma}

\newcommand{\Id}[1]{\color{blue} \color{black}}

\makeatletter
\newcommand{\inlineitem}[1][]{%
	\ifnum\enit@type=\tw@
	{\descriptionlabel{#1}}
	\hspace{\labelsep}%
	\else
	\ifnum\enit@type=\z@
	\refstepcounter{\@listctr}\fi
	\quad\@itemlabel\hspace{\labelsep}%
	\fi}
\makeatother

\title{\LARGE \bf
Essentially decentralized conjugate gradients
}

		\author{Alexander Engelmann and Timm Faulwasser}
	\thanks{AE and TF are both with the Institute for Energy Systems, Energy Efficiency and Energy Economics, TU Dortmund University, Germany.
	Parts of this work have been conducted while AE was with the Institute for Automation and applied Informatics (IAI), Karlsruhe Institute of Technology, Karlsruhe, Germany in the PhD thesis \cite{Engelmann2020e}.
		{\tt\small alexander.engelmann@tu-dortmund.de, timm.faulwasser@ieee.org}%
}

\begin{document}

\maketitle
\thispagestyle{empty}
\pagestyle{empty}

\begin{abstract}
Solving structured systems of linear equations in a non-centralized fashion is an important step in many distributed optimization and control algorithms.
	Fast convergence is required in manifold applications.
	Known decentralized algorithms, however, typically exhibit asymptotic convergence  at a linear rate.
	This note proposes an essentially decentralized variant of the Conjugate Gradient algorithm (d-CG).  The proposed method exhibits a practical superlinear  convergence rate and comes with a priori computable finite-step convergence guarantees. 
	In contrast to previous works, we consider sum-wise decomposition instead of row-wise decomposition which enables application in multi-agent settings.
	We illustrate the performance of d-CG on problems from sensor fusion and compare the results to  the widely-used Alternating Direction Method of Multipliers. 	
\end{abstract}

\section{Introduction}
Sparse and structured linear equations  occur frequently as subproblems in different contexts of optimization and control algorithms. 
In many of these algorithms, the coefficient matrix is a sum of individual coefficient matrices, each of which is assigned to one subsystem and carries a sparsity pattern related to the underlying graph structure.
Examples are sensor networks~\cite{Xiao2005,Schizas2008}
{and structure exploiting distributed or parallel optimization algorithms based on the Schur complement \cite{Zavala2008,Kardos2020c,Engelmann2020c}.}

	If the coefficient matrix is positive semi-definite, solving the linear equations can equivalently be formulated as a separable convex optimization problems rendering  standard decentralized  algorithms such as the Alternating Direction Method of Multipliers~(ADMM)~ \cite{Shi2014,Makhdoumi2017} or Jacobi iterations based on Schwarz decomposition~\cite{Shin2020}  applicable.
{Also decentralized subgradient schemes such as EXTRA, NEXT, DIGing or PANDA~  \cite{Shi2015,DiLorenzo2016,Nedic2017,Maros2018} are applicable, which converge even under asynchrony or under  time-varying communication graphs.}
{However, none of these algorithms  exhibits exact convergence guarantees in an a priori computable number of iterations.}

{This paper addresses this gap and presents an essentially decentralized Conjugate Gradient algorithm (d-CG) which solves systems of linear equations, where the coefficient matrix is a sum over individual subsystems.
	The proposed algorithms avoids  \textit{central computation} entirely, but  requires a rather small amount of \textit{global communication}.
	While distributed optimization typically allows for  a preferably small central computation and communication \cite{Bertsekas1989}, decentralized optimization allows neither for central communication nor for central computation~\cite{Nedic2018}.
	Hence, we refer to d-CG as being \textit{essentially decentralized}---a sub-category in between distributed and decentralized optimization.
}
We show that d-CG converges to the {exact} solution in  a finite, a priori computable number of iterations. Moreover, we demonstrate practical superlinear convergence, which is one of the fastest convergence rates currently known for distributed and decentralized algorithms.

{In contrast to our approach, classical parallel conjugate gradient methods such as  \cite{RadicatiDiBrozolo1989,Helfenstein2012,Lof2004} decompose the coefficient matrix {row-wise}, i.e., each subsystem corresponds to a specific number of rows in the coefficient matrix. 
	This renders its application to multi-agent systems rather difficult, e.g. in context of optimization algorithms based on Schur complements with a given decomposition structure.\footnote{{Other row-wise decomposition approaches are proposed in \cite{Mou2015,Shi2017}, which are guaranteed to converge even under asynchrony and communication delays. These approaches do not rely on CG.}}}
The idea of distributed conjugate gradient algorithms has been previously considered in the context of bi-level distributed algorithms for non-convex problems \cite{Engelmann2020c}. 
In contrast, the present paper focuses on sparsity exploitation and it avoids  pre-computation steps.

The remainder of this paper is structured as follows:
In \autoref{sec:probForm}, we introduce a sparsity encoding framework and briefly recall the centralized conjugate gradient algorithm.
We continue by developing the essentially decentralized conjugate gradient method in \autoref{sec:d-CG} based on the sparsity framework from \autoref{sec:probForm}.
\autoref{sec:App} demonstrates the effectiveness of our approach via numerical examples and in comparison to ADMM.

\textit{Notation:} A weighted semi-norm is written as  $\|x\|_{W}\doteq\sqrt{x^\top W x}$ with $W\succeq 0$.
For  a sequence of matrices~$\{W_i\}_{i \in \mathcal{S}}$,
$\operatorname{blkdiag}\,\{W_i\}_{i \in \mathcal{S}}$ denotes block-diagonal concatenation and $\operatorname{vec}\,\{W_i\}_{i \in \mathcal{S}}$ refers to vertical concatenation.
{We denote the $l$th row of a matrix $M\in \mathbb{R}^{n\times m }$ by $[M]_l$, and $[M]_{kl}$ is the $k-l$ entry of $M$.}

\section{Preliminaries and Problem Formulation} \label{sec:probForm}
Consider a set of subsystems/agents $\mathcal{S} \doteq \{1,\dots,|\mathcal S|\}$, which aims at cooperatively solving
\begin{align} \label{eq:SchurComp}
	\underbrace{\left  (\sum_{i \in \mathcal{S}} S_i \right  )}_{\doteq S} \lambda =\underbrace{\sum_{i \in \mathcal{S}} s_i. }_{\doteq s}
\end{align}
All coefficient matrices $S_i \in \mathbb{R}^{m\times m}, i \in \mathcal{S}, $ are symmetric positive \emph{semi-definite}, and we assume that each $S_i=S_i^\top$ carries a  sparsity structure, i.e., specific  rows and columns of $S_i$ are zero.
We collect all indices for which $S_i$ has non-zero rows in the set 
\begin{align} \label{eq:defCi}
	\mathcal{C}_i\doteq\{l\in \{1,\dots,m\} \; | \; [S_i]_{l} \neq 0\}.
\end{align}
{Due to $S_i=S_i^\top$, the set $\mathcal{C}_i$ also encodes non-zero columns. The set of neighbours of system $i$  is
	$
	\mathcal{N}_i\doteq \{j\in \mathcal{S} \; |\;  \mathcal{C}_i \cap \mathcal{C}_j \neq \emptyset\}.
	$}
{Hence, \eqref{eq:SchurComp} is implicitly defined over an undirected communication graph $G=(\mathcal{S},\mathcal{E})$
	with nodes set $\mathcal{S}$ and edge set $\mathcal{E} \doteq  \{(i,j)\;|\; \mathcal{C}_i \cap \mathcal{C}_j \neq \emptyset \}$.}

\subsection*{The Centralized Conjugate Gradient Method} 
{Centralized CG for solving \eqref{eq:SchurComp} is summarized in Algorithm~\ref{alg:CG}.
	It  dates back to \cite{Hestenes1952} and is discussed in many standard sources such as  \cite[Ch. 5.1]{Nocedal2006}.
}
\begin{algorithm}[t]
	\small 
	\caption{Conjugate Gradients for problem \eqref{eq:SchurComp} }
	\textbf{Initialization: $r^0=p^0=s -S  \lambda^0$.}\\
	\textbf{Repeat until: $r^n= 0$}
	\begin{subequations} \label{eq:CGiter}
		\begin{align}
			\alpha^n &= \frac{r^{n\top}r^n}{p^{n\top}Sp^n} \label{eq:CGalpha}\\
			\lambda^{n+1} &=  \lambda^n + \alpha^n p^n \label{eq:CGlam}\\
			r^{n+1} &= r^n-\alpha^nSp^n\label{eq:CGr}\\
			\beta^n &= \frac{r^{n+1 \top} r^{n+1}}{r^{n\top}r^n} \label{eq:CGbeta} \\
			p^{n+1} &= r^{n+1} + \beta^n p^n \label{eq:CGp} \\
			n & \leftarrow n+1 \notag 
		\end{align}
	\end{subequations}
	\vspace{-.4cm}
	\label{alg:CG}
\end{algorithm} 
We recall a convergence result  {from \cite{Hayami2020}} for systems with positive semi-definite matrices $S$ and $s\in \operatorname{range}(S)$, which extends the standard convergence analysis of CG.
{Note that if $s\notin \operatorname{range}(S)$, problem \eqref{eq:SchurComp} does not have solution.}
Semi-definite $S$ arise frequently in network settings, as redundant constraints induced by the communication graph lead to zero eigenvalues of $S$, cf. \autoref{sec:distEst}.
\begin{thm}[Convergence of CG {\cite{Hayami2020}}] \label{thm:cgConv}
	{Consider problem \eqref{eq:SchurComp}.}
	If $S\succeq0$, $\operatorname{range}(S)\neq \{0\}$ and $s\in \operatorname{range}(S)$, Algorithm~\ref{alg:CG} converges to a solution of \eqref{eq:SchurComp}, $ \lambda^\star$, and the convergence-rate estimate
	\begin{align}\label{eq:est2}
		\| \lambda^{n+1} -  \lambda ^\star\|_{S}^2
		\leq 
		2\left ( \frac{ \sqrt{\kappa} - 1}{\sqrt{\kappa} + 1}\right)^n
		\| \lambda^{0} -  \lambda ^\star\|_{S}^2
	\end{align}
	holds, where $\kappa = \omega_{\mathrm{max}}/ \omega_{\mathrm{min}}$ is the ratio of the largest and smallest non-zero eigenvalues of $S$.
	Moreover, CG terminates after at most $m-n_0$ steps, where $n_0$ is the number of zero-eigenvalues of $S$. \hfill $\square$
\end{thm}
{	Despite $S\succ 0$, the convergence estimates are useful since $ \lambda^\star \not\in\operatorname{null}(S)$. Hence the bound characterizes convergence in $\operatorname{range}(S)$. For the sake of completeness, we present an alternate proof to one from \cite{Hayami2020} in the Appendix.}

{\begin{rem}[CG convergence bounds for semi-definite $S$]\label{rem:semDefS}
		First convergence rate bounds for CG and positive definite $S$ appear in \cite{VanDerSluis1986}.
		The analysis therein can immediately be extended to the semi-definite case if $s \in \operatorname{range}(S)$, which is recognized in \cite{Axelsson1994,Kaasschieter1988}.
		For an explicitly computable bound in terms of  min/max (non-zero) eigenvalues using a simplified analysis see \cite{Hayami2020}.	\hfill $\square$
\end{rem}}

\section{Decentralized Conjugate Gradients} \label{sec:d-CG}
The main contribution of this note is a essentially decentralized conjugate gradient algorithm which is presented next.

\subsection*{Sparsity Exploitation}
{We consider projection matrices  $I_i \in \mathbb{R}^{|\mathcal{C}_i|\times m}$ to eliminate zero-rows and zero-columns in $S_i$.
	This enables fast convergence and it simplifies message-passing between subsystems.}
We construct these projections as the vertical concatenation of standard Euclidean basis  vectors 
$e_i^\top\doteq (0,\;\dots,1
,\dots,0) \in \mathbb{R}^{m}$
indexed by $\mathcal{C}_i$ from \eqref{eq:defCi}, i.e.,
\begin{align} \label{eq:Idef}
	{I_i\doteq \operatorname{vec}\,\{e_j^\top\}_{ j \in \mathcal{C}_i}.}
\end{align}
Moreover, we consider diagonal matrices
\begin{align} \label{eq:LamDef}
	\Lambda_i\doteq I_i \left ( \textstyle \sum_{j \in \mathcal{S}} I_j^{ \top} I_j^{\phantom \top} \right ) I_i^{ \top}
	= I_i^{\phantom \top} \Lambda I_i^{ \top},
\end{align}
with $ \Lambda \doteq \sum_{j \in \mathcal{S}} I_j^{ \top} I_j^{\phantom \top}$.
Next we formalize properties of $\Lambda_i$ and $ S_i$ with respect to multiplication with $I_i$.
{To this end, $E_i \in \mathbb{R}^{|\mathcal{C}_i| \times |\mathcal{C}_i| }$ is defined 
	as a diagonal matrix with $[E_i]_{jj} = 1$ if $j \in \mathcal{C}_i$ and $[E_i]_{jj} = 0$ if $j \not\in \mathcal{C}_i$.}
\begin{lem}[Properties of $S_i$, $ I_i$ and $\Lambda_i$] \label{lem:propS}
	Consider $S_i$ from \eqref{eq:SchurComp},  $ I_i$ from \eqref{eq:Idef}, and $\Lambda_i$  from \eqref{eq:LamDef} with  $\mathcal C_i$ from \eqref{eq:defCi}.
	Then, the following properties hold:
	\begin{align*}
		&\text{(P1)}\; I_i^{\top} I_i^{\phantom \top}=E_{i}, 
		&&\text{(P2)}\; I_i^{\top} I_i^{\phantom \top}  S_i =  S_i , \\
		&\text{(P3)}\; I_i^{\top} I_i^{\phantom \top}  s_i = s_i, 
		&&\text {(P4)}\; I_i^{\top} I_i^{\phantom \top}I_i^{\top}=I_i^{\top}, \\
		&\text{(P5)}\; \sum_{i \in \mathcal{S}}I_i^{\top}\Lambda_i^{-1}I_i^{\phantom \top}=I. &&& \hspace*{-1.05cm} \square
	\end{align*}			
\end{lem}
\begin{proof}
	(P1): Follows immediately from $e_i^\top e^{\phantom \top}_j = \delta_{ij}$, where $\delta_{ij}$ is the Kronecker delta. (P2): By (P1) and the definition of $\mathcal{C}_i$, we have  $I_i^{\top} I_i^{\phantom \top}  S_i = E_{i} S_i  =  S_i$; (P3)  and (P4) follow similarly. 
	(P5): Since $\Lambda$ is diagonal, and by definition of $I_i^{ \phantom \top}$ as the concatenation of unit vectors, we have $ \Lambda^{-1}_i = (I_i^{\phantom \top}  \Lambda \,I_i^{ \top})^{-1}=I_i^{\phantom \top}  \Lambda^{-1} \,I_i^{ \top}$.
	Using (P1)-(P4) and $AB=BA$ for diagonal matrices yields 
	\begin{multline*}
		\sum_{i \in \mathcal{S}} I_i^{\top}  \Lambda^{-1}_i\, I_i^{ \phantom \top} = \sum_{i \in \mathcal{S}} I_i^{\top}  I_i^{\phantom \top}  \Lambda^{-1} \,I_i^{ \top}\, I_i^{ \phantom \top}  \\
		= \sum_{i \in \mathcal{S}} I_i^{\top}  I_i^{\phantom \top} I_i^{ \top}\, I_i^{ \phantom \top} \Lambda^{-1} 
		= \sum_{i \in \mathcal{S}} I_i^{\top}  I_i^{\phantom \top} \Lambda^{-1} 
		=\; \Lambda\, \Lambda^{-1}.
	\end{multline*}
\end{proof}

Next, we derive the d-CG algorithm based on \autoref{lem:propS}.
The main idea of d-CG is to exploit sparsity in \eqref{eq:CGalpha}-\eqref{eq:CGp} via the projections $I_i$.
We start by defining local equivalents for the CG iterates $\lambda, r$ and $p$,
\begin{subequations}\label{eq:CGdef}
	\begin{align} 
		\lambda_i \doteq I_i  \lambda, \;\;  r_i \doteq I_i r, \;\; p_i \doteq I_i p, \;\; \text{for all} \;\; i \in \mathcal{S}.
	\end{align}
	Moreover, let 
	\begin{equation}
		I_{ij}\doteq I_i^{\phantom \top} \hspace{-1.2mm} I_j^{\top} \hspace{-1.2mm},\;\; \hat S_i \doteq I_i^{\phantom \top} \hspace{-1.2mm}   S_i^{\phantom \top} \hspace{-1.2mm}   I_i^{\top}\hspace{-1.2mm},\;\; \hat {s}_i \doteq I_is_i,\; \text{for all} \; i \in \mathcal{S}.
	\end{equation}
\end{subequations}
With these definitions, the proposed  d-CG scheme is summarized in Algorithm~\ref{alg:d-CG}.
\begin{prop}[Convergence bounds for d-CG] \label{prop:convDCG}
	{Consider problem~\eqref{eq:SchurComp} under the assumptions of \autoref{thm:cgConv} and the variable definitions~\eqref{eq:CGdef}.
		Then, for any initialization $\lambda_i^0, i\in \mathcal{S}$, Algorithm~\ref{alg:d-CG} exhibits the convergence properties from \autoref{thm:cgConv}.} \hfill $\square$
\end{prop}
\begin{algorithm}[t]
	\small 
	\caption{Ess. decentralized CG  for problem \eqref{eq:SchurComp} }
	\textbf{Initialization for all $i \in \mathcal{S}$: $ \lambda^0_i$, $r_i^0=p_i^0= \sum_{j \in \mathcal{N}_i}I_{ij} \hat s_j - \sum_{j \in \mathcal{N}_i} I_{ij}^{\phantom \top} \hat S_j \lambda_j^0$,  $\eta^{0}_i= r_i^{0\top } \Lambda_i^{-1} r_i^0$ and $\eta^{0}= \sum_{i \in \mathcal{S}} \eta_i^{0}$.}\\
	\textbf{Repeat until $r_i^k= 0$ for all $i \in \mathcal{S}$:}
	\begin{subequations}
		\begin{align} 
			&\sigma^{n}_i = p^{n \top }_i \hat S_i p^{n}_i;\;\;  u_i^{n}= \hat S_i p_i^{n}; \hspace{-2.3cm} & \hfill \text{(local)} \label{step:dCG6} \\
			& \textstyle{  \sigma^{n} = \sum_{i \in \mathcal{S}} \sigma_i^{n} } \hspace{-.8cm} & \text{ (scalar global sums)} \label{step:dCG1}\\
			&\textstyle{ r^{n+1}_i = r_i^n - \frac{\eta^{n}}{\sigma^{n}} \sum_{j \in \mathcal  N_i} I_{ij} u_j^n} &\hfill \text{ (neighbor-to-neighbor)} \label{step:dCG2} \\
			&  \textstyle{\eta^{n+1}_i= r_i^{n + 1\top } \Lambda_i^{-1} r_i^{n+ 1} } & \text{(local)} \label{step:dCG3}\\
			& \textstyle{\eta^{n+1}= \sum_{i \in \mathcal{S}} \eta_i^{n+1} } \hspace{-1.2cm} &\hfill \text{ (scalar global sum)} \label{step:dCG4} \\
			&\textstyle{p_i^{n+1}= r_i^{n+1} + \frac{\eta^{n+1}}{\eta^n} p_i^n }; \;\; \textstyle{\lambda_i^{n+1} = \lambda_i^n + \frac{\eta^{n}}{\sigma^n} p_i} \hspace{-1.8cm} &\text{(local)} \label{step:dCG5} \\
			&n \leftarrow n+1 & \notag
		\end{align}
	\end{subequations} \label{alg:d-CG}
	\vspace{-.5cm}
\end{algorithm} 
\begin{proof}
	We give a constructive proof, i.e., we show that the iterates of Algorithm~\ref{alg:CG} and Algorithm~\ref{alg:d-CG} are equivalent. 
	First we decompose \eqref{eq:CGalpha}. 
	Consider $\alpha^n \doteq \frac{\eta^n}{\sigma^n}$ and $\eta^n \doteq r^{n\top}r^n$.
	From \autoref{lem:propS}, we  have that 
	\begin{align*}
		\eta^n = r^{n\top}r^n = r^{n\top}I\;r^n =
		r^{n\top}\Big (\sum_{i \in \mathcal{S}} I_i^{\top}  \Lambda^{-1}_i\, I_i^{ \phantom \top}\Big )\;r^n.
	\end{align*}
	Consider  $\eta_i\doteq r_i^{n\top}\Lambda_i^{-1} r^n_i $, then  we obtain via  \eqref{eq:CGdef}
	\begin{align*}
		\eta^n = r^{n\top}r^n = \sum_{i \in \mathcal{S}}r_i^{n\top}\Lambda_i^{-1} r^n_i = \sum_{i \in \mathcal{S}} \eta_i,
	\end{align*}
	where $\eta_i$ can be computed locally in each subsystem.
	This yields \eqref{step:dCG3} and \eqref{step:dCG4}. 
	By construction of $S, p_i$ and invoking \autoref{lem:propS}, the denominator of  \eqref{eq:CGalpha} can be written as 
	\begin{align*}
		\sigma^n
		&= p^{n\top}\sum_{i \in \mathcal{S}}S_ip^n
		= p^{n\top}\sum_{i \in \mathcal{S}}I_i^{\top} I_i^{ \phantom \top}S_iI_i^{\top} I_i^{ \phantom \top}p^n\\
		&= \hspace{-1mm}\sum_{i \in \mathcal{S}} p^{n\top}_i\hat S_ip^n_i
		=\hspace{-1mm}\sum_{i \in \mathcal{S}} \sigma_i^n,
	\end{align*}
	with  $\hat S_i$ from \eqref{eq:CGdef}. 
	This yields \eqref{step:dCG6} (first equation) and \eqref{step:dCG1}.
	Note that the sums $\sum_{i \in \mathcal{S}} \eta_i$ and $ \sum_{i \in \mathcal{S}} \sigma_i^{n}$ require \emph{global} communication of two scalars $(\eta_i,\sigma_i)$ per subsystem $i \in \mathcal{S}$.\footnote{This can be achieved via hopping protocols.}
	
	Next, consider \eqref{eq:CGlam} and \eqref{eq:CGp}.
	Left multiplying both equations by $I_i^{ \phantom \top}$  yields
	\begin{align*}
		\lambda^{n+1}_i =  \lambda^n_i + \textstyle \frac{\eta^n}{\sigma^n} p^n_i \quad \text{and} \quad
		p^{n+1}_i = r^{n+1}_i + \textstyle \frac{\eta^{n+1}}{\eta^{n}}  p^n_i
	\end{align*}
	comprising  \eqref{step:dCG5}  in Algorithm~\eqref{alg:d-CG}.
	Finally, we decompose \eqref{eq:CGr} which requires neighbor-to-neighbor communication due to $Sp^n$.
	\autoref{lem:propS}, and the definitions of $p_i$ and $\hat S_j$ yield
	\begin{align*}
		r^{n+1} &= r^n-\frac{\eta^n}{\sigma^n} Sp^n = 
		r^n-\frac{\eta^n}{\sigma^n} \sum_{j \in \mathcal{S}} I_j^{\top}  \hat S_j p^n_j.
	\end{align*}
	Left multiplying by $I_i^{\phantom \top}$ gives
	\begin{align*}
		r^{n+1}_i =  r^n_i-\frac{\eta^n}{\sigma^n}  \sum_{j \in \mathcal N_i} I_{ij}\hat S_j p^n_j= r^n_i-\frac{\eta^n}{\sigma^n} \sum_{j \in \mathcal  N_i} I_{ij}u_j^n,
	\end{align*}
	with  $u_j^n\doteq\hat S_j p^n_j$ and $I_{ij}$ from \eqref{eq:CGdef}. 
	Since $I_{ij} \hspace{-.8mm}=\hspace{-.6mm}0 $ if $j \neq \mathcal  N_i$, the summation is  over  the neighbors of subsystem~$i$.
	Moreover, $u_j^n=\hat S_j p^n_j$ can be computed locally by each subsystem.
	This  yields \eqref{step:dCG6} (second part) and \eqref{step:dCG2}.
	Multiply the initial condition in Algorithm~\ref{alg:CG}  by  $I_i$ and use (P2) to get
	\begin{align*} 
		r_i^0=p_i^0 
		= \sum_{j \in \mathcal{N}_i}I_{ij} \hat s_j - \sum_{j \in \mathcal{N}_i} I_{ij}^{\phantom \top} \hat S_j \lambda_j^0. 
	\end{align*}
	Hence the iterates of d-CG and CG are equivalent. 
\end{proof}
{We remark that if $\lambda_i^0 \not = 0$, one needs one extra neighbor-to-neighbor communication.}
{	\begin{rem}[Assumptions on the communication graph] \label{rem:commStruct}
		Recall that d-CG relies on the undirected communication graph $G$ induced by  \eqref{eq:defCi}. Notice that  connectedness of $G$ is not required as lack thereof implies block-diagonal structure in \eqref{eq:SchurComp}.
		Moreover, matching a pre-defined communication structure is possible via adjusting the problem formulation, cf. \autoref{sec:App}.\hfill $\square$
\end{rem}}

\subsection*{Discussion}
For the sake of comparison, the Appendix details a variant of decentralized ADMM (d-ADMM) based on the same sparsity model used above.
\autoref{tab::convPropCGAL} compares properties of d-ADMM and d-CG.
An advantage of d-ADMM is that it maintains only  two local variables, whereas d-CG  maintains four local variables.
{Moreover, d-ADMM is fully decentralized and does not require global communication.}
On the other hand, d-CG achieves  finite-step convergence to the {exact} solution.
Moreover, the practically observed convergence rate is {superlinear}, whereas d-ADMM or comparable algorithms achieve an at most linear or even {sublinear} rate \cite{Nedic2018}. 
{A further advantage of d-CG is that it does not require tuning while the realized convergence speed of ADMM  substantially depends on the chosen parameters.}

\begin{table}
	\centering
	\caption{Properties of d-CG and d-ADMM for \eqref{eq:SchurComp}.}
	\begin{tabular}{rlll}
		\toprule
		convergence rate	&	d-CG & d-ADMM  \\
		\midrule
		theoretical&  $m$-step & linear/sublinear   \\
		practical &  superlinear   & linear/sublinear \\
		tuning required & no & yes \\
		local variables & four & two \\
		\bottomrule
	\end{tabular}
	\label{tab::convPropCGAL}
\end{table}

\begin{rem}[Difference to d-ADMM and d-CG from \cite{Engelmann2020c}]
	{In contrast to the versions of d-CG and d-ADMM from \cite{Engelmann2020c}, we  rely on} the sparsity results of \autoref{lem:propS}.
	{This approach eliminates any precomputation in d-CG,} and for d-ADMM this {often} gives strong convexity of the  local objectives $f_i$ in \eqref{eq:consLam2}, which in turn implies linear instead of sublinear convergence, cf. \cite{Yang2016,Wei2013,Nedic2018a}.\hfill $\square$
\end{rem}

\section{Applications} \label{sec:App}

The introduction claimed that problem \eqref{eq:SchurComp} occurs in many applications.
{Indeed examples are}   distributed non-convex second-order algorithms such as distributed interior-point methods \cite{Zavala2008,Pakazad2017,Kardos2020c}, distributed SQP methods \cite{Kozma2014}, and ALADIN \cite{Houska2016}.
These methods have in common that a KKT system is solved centrally in each iteration, which is a substantial hurdle for   decentralization.
Under standard regularity and second-order conditions,  KKT systems can be written in form of \eqref{eq:SchurComp}. Thus they can be solved via d-ADMM or d-CG---this way obtaining decentralized variants of these algorithms. 
A  blueprint approach reads as follows:
\begin{enumerate}
	\item Reformulate an optimization problem in affinely-coupled separable form:\footnote{This is always possible using auxiliary variables \cite{Engelmann2020c}.}
	\begin{subequations} \label{eq:sepForm}
		\begin{align} 
			\min_{x_1,\dots,x_S} \; \sum_{i \in \mathcal{S}} &\,f_i(x_i) \\
			\text{subject to }\quad  g_i(x_i)&=0 \; \;|\; \; \kappa_i, & \forall i \in \mathcal{S}, \label{eq:sepProbGi} \\
			\sum_{i \in \mathcal{S}} A_ix_i &= b\; \;|\; \; \lambda.\label{eq:consConstr}
		\end{align}
	\end{subequations}
	\item At the current iterate $(x^k,\kappa^k,\lambda^k)$, compute a quadratic approximation of \eqref{eq:sepForm}  where $B_i^k$ and $G_i^k$ are Hessian and Jacobian approximations.
	This yields the KKT system
	\begin{align} \label{eq:saddleProb}
		\begin{pmatrix}
			H & A^\top \\
			A & D\phantom{^\top}
		\end{pmatrix}
		\begin{pmatrix}
			p \\ \lambda
		\end{pmatrix}
		=
		\begin{pmatrix}
			-h \\ \phantom{-}d
		\end{pmatrix},
	\end{align}
	where $p^\top =(x^\top,\kappa^\top )$, $A=(A_1,\dots,A_S)$, and
	{	\begin{align*}
			H = {\operatorname{blkdiag}} 
			\left \{
			\begin{pmatrix}
				B_i & G_i^\top  \\
				G_i & 0\phantom{_i^\top}
			\end{pmatrix}
			\right \}_{i \in \mathcal{S}}
			\hspace{-.5cm},
			~h  = \operatorname{vec}
			\left \{
			\begin{pmatrix}
				g_i \\0
			\end{pmatrix}
			\right\}_{i \in \mathcal{S}}.
		\end{align*}
	}
	\item Compute the Schur complement---i.e., solve the first row in \eqref{eq:saddleProb} for $p$ and insert it to the second---to obtain   \eqref{eq:SchurComp}.
\end{enumerate}
This  {blueprint} reformulation substantially reduces the dimension of the KKT system \eqref{eq:saddleProb}---thus it often improves run-time. Moreover, in contrast to applying decentralized algorithms directly to~\eqref{eq:saddleProb} \cite{Shin2020}, one obtains a  system with positive-definite coefficient matrix.

\begin{rem}[Requirements on \eqref{eq:saddleProb}]
	The above procedure requires that all $B_i$ are invertible. 
	This is  can be enforced by a slightly stronger version of the Second-Order Sufficient Condition for \eqref{eq:sepForm}, cf. \cite[Ass 1]{Engelmann2020c}, or via regularization. \hfill $\square$
\end{rem}

\subsection*{Example -- Decentralized sensor fusion} \label{sec:distEst}
We apply the above blueprint to a sensor fusion problem~\cite{Xiao2005,Rabbat2004a,Schizas2008}.
The goal is to  estimate the parameters $\theta \in \mathbb{R}^{n}$ collaboratively  by minimizing the variance of the estimate.
The measurements for each sensor $i \in \mathcal{S}$ are given by
$
y_i = M_i \theta + \mathsf{v}_i$,
with $\mathsf{v} \sim \mathcal{N}(0,1)$,  $y_i \in \mathbb{R}^{n_{y}}$, and $\theta \in \mathbb{R}^{n_\theta}$.
We assume that all $M_i$ have full row rank.
Thus, the maximum-likelihood estimate of $\theta$  is
\begin{align} \label{eq:MLE}
	\theta^\star = \underset{\theta}{\operatorname{argmin}} \sum_{i \in \mathcal{S}} \frac{1}{2}\|y_i - M_i \theta\|_2^2.
\end{align}
{To enforce a communication graph one commonly adds additional consensus constraints, each of which corresponds to one communication link between sensors \cite{Schizas2008}.}
{Consider a connected communication graph} $G=(\mathcal{S},\mathcal E)$.
Then, we reformulate \eqref{eq:MLE} as
\begin{subequations}\label{eq:consEst}
	\begin{align} 
		\theta^\star = \underset{\theta_1,\dots,\theta_S}{\operatorname{argmin}}\;\;& \sum_{i \in \mathcal{S}} \frac{1}{2}\|y_i - M_i \theta_i\|_2^2 \\
		\text{ s.t.}\quad & \theta_i = \theta_j \; \text{ for all } \; (i,j) \in \mathcal E. \label{eq:consCnstr}
	\end{align}
\end{subequations}
{corresponding to step 1) of the blueprint.}
Observe that \eqref{eq:consCnstr} is equivalent to
$\sum_{i \in \mathcal{S}} A_i \theta_i=0$ with  $A=I_G \otimes I $, where $I_G$ is the incidence matrix of  $G$. 
The Lagrangian of \eqref{eq:consEst} reads
\begin{align*}
	\mathcal L(x,\lambda)= \sum_{i \in \mathcal{S}} \frac{1}{2}\|y_i - M_i \theta_i\|_2^2 + \lambda^\top \sum_{i \in \mathcal{S}}A_i \theta_i.
\end{align*}
Thus, the first-order optimality conditions of \eqref{eq:consEst} are
\begin{align}
	M_i^\top (y_i-M_i \theta_i) + A_i^\top \lambda = 0, \; \text{ and } \;\;
	\sum_{i \in \mathcal{S}} A_i\theta_i = 0\label{eq:2ndRow}
\end{align}
{corresponding to step 2).\footnote{Note that for \eqref{eq:consEst} LICQ does not hold in general. However, Slater's condition is always satisfied and thus we can use the KKT conditions.  }}
Solving \eqref{eq:2ndRow} yields $\theta_i = (M_i^\top M_i)^{-1} (M_i^\top y_i + A_i^\top \lambda)$.
{Step 3) gives
	\begin{align} \label{eq:SchurNet}
		\left (\sum_{i \in \mathcal{S}} A_i(M_i^\top M_i)^{-1}  A_i^\top \right  )\lambda 
		=s, 
	\end{align}
	with $s \doteq -\sum_{i \in \mathcal{S}} A_i(M_i^\top M_i)^{-1} M_i^\top y_i$.
}
\begin{lem}[Verifying the assumptions of \autoref{prop:convDCG}]
	Suppose $M_i, i \in \mathcal{S}$ have full row rank. Then $s\in \operatorname{range}(S)$ in \eqref{eq:SchurNet}. \hfill $\square$
\end{lem}
\begin{proof}
	The definition of $s$ implies  $s\in \operatorname{range}(A)$.
	$S$ on the left-hand side of \eqref{eq:SchurNet}  involves multiplication with $A$. Hence, if $M^\top M$ has full rank, $\operatorname{range}(S)=\operatorname{range}(A)$.
\end{proof}

\subsection*{Numerical results}
We consider a network of $\mathcal S=\{1,\dots,10\}$ sensors with different communication topologies
\begin{figure}
	\centering
	\includegraphics[trim=10 40 10 10,clip,width=.8\linewidth]{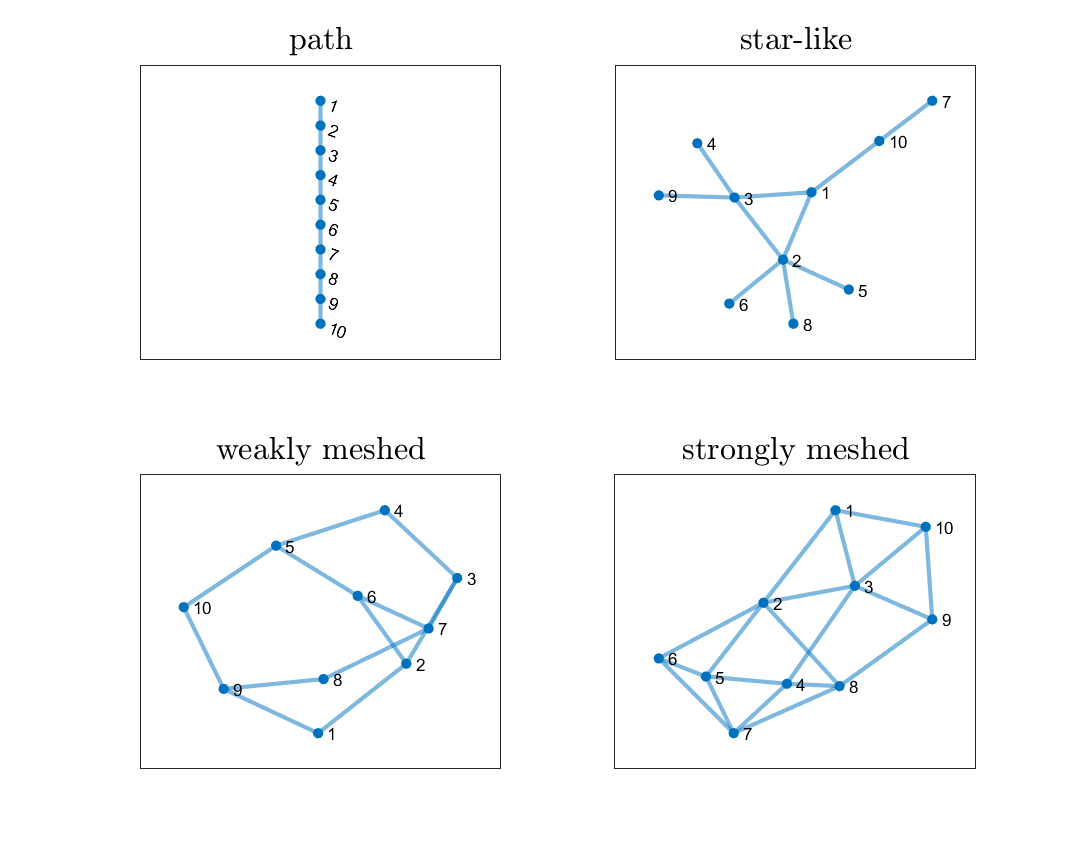}
	\caption{Investigated communication topologies.}
	\label{fig:graphs}
\end{figure}
such as a path graph, a star star-like graph, and two meshed graphs, cf.  \autoref{fig:graphs}.
The measurement matrix $M$ is randomly generated with  all matrix entries drawn from a uniform distribution on~$[0,1]$.
{We have $n_y=10$ measurements per sensor and $n_\theta = 10$ parameters.}
The measurement noise is Gaussian with $\mathsf{v} \sim \mathcal{N}(0,10^{-3})$.

The numerical performance of d-CG and d-ADMM is shown in \autoref{fig:sweep}.
\autoref{fig:rhoSwp} depicts the number of iterations required for reaching an accuracy of $\|S  \lambda -s\| < 10^{-5}$. 
As one can see, for all considered network topologies, d-CG is about 2-5 times faster than ADMM even in the best-tuned case.
This also holds true for a larger and randomly generated meshed graph with $\mathcal |S| =1000$ sensors and $\mathcal |\mathcal E| = 1500$ edges, where d-CG requires 29 iterations to converge to $\|S  \lambda -s\| < 10^{-4}$ and ADMM requires 86 iterations in the best-tuned case.
Moreover, one can observe that tuning of d-ADMM is quite difficult: the optimal tuning parameter $\rho$ varies with the network topology and significantly affects convergence speed.
Finally, \autoref{fig:conv} shows the distance to a minimizer measured in the weighted semi-norm $\|\lambda-\lambda^\star\|_{S}$ for the strongly meshed graph.
Observe that the convergence rate bound \eqref{eq:est2} is met.
Moreover, d-CG converges to a (numerically) exact solution after $69$ iterations, which is strictly less than  $n_o=90$  non-zero eigenvalues predicted by \autoref{thm:cgConv} and \autoref{prop:convDCG}.

\begin{figure}
	\centering
	\begin{subfigure}{\linewidth}
		\centering
		{\includegraphics[trim=0 0 0 10,width=0.55\linewidth]{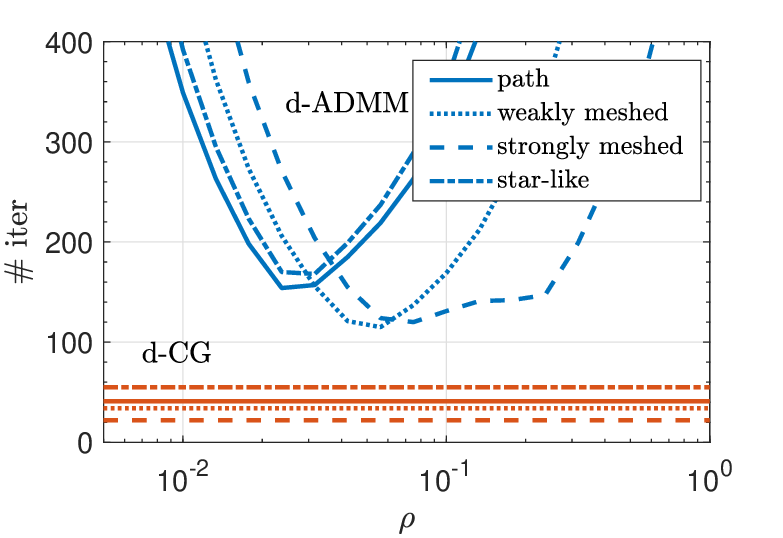}}
		\caption{ Required $\#$iterations  for reaching $\|S  \lambda -s\| < 10^{-5}$ for d-CG~(red) and d-ADMM (blue) and different communication topologies depending on the tuning parameter $\rho$.}
		\label{fig:rhoSwp}
	\end{subfigure}
	\begin{subfigure}{\linewidth}
		\centering
		\includegraphics[trim=0 0 0 10,clip,width=.55\linewidth]{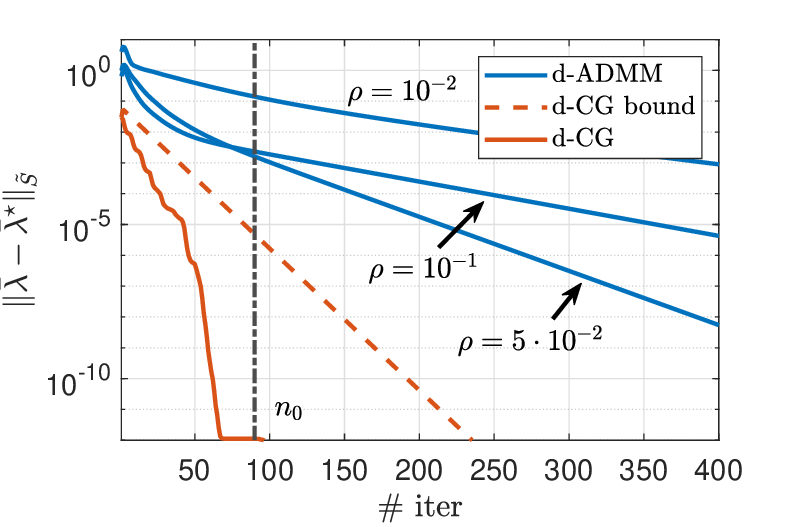}
		\caption{Convergence of d-CG and d-ADMM for the strongly meshed graph with convergence bound \eqref{eq:est2} for d-CG. }
		\label{fig:conv}
	\end{subfigure}
	\caption{Numerical performance of d-CG and d-ADMM.}
	\label{fig:sweep}
	\vspace{-0.74cm}
\end{figure}

\section{Summary}
This note presented an {essentially} decentralized version of the conjugate gradient algorithm for solving structured systems of linear equations over networks.
{The proposed algorithm achieves exact convergence in a finite number of steps, and the number of steps is a priorily computable. In examples, it has shown practically superlinear convergence.}
Future work will investigate the design  of communication topologies to optimize convergence rate.

\appendix
\section*{Proof of \autoref{thm:cgConv}} \label{sec:proof}
We analyze the CG iterates \eqref{eq:CGalpha}-\eqref{eq:CGp} in the kernel and in the range of $S$ separately.
To this end, denote vectors in the kernel of $S$ by subscript $(\cdot)_N$ and vectors in the range space of $S$ by subscript  $(\cdot)_R$. 
Since $\operatorname{null}(S)$ is the orthogonal complement of $\operatorname{range}(S^\top)$, and $\operatorname{range}(S^\top)=\operatorname{range}(S)$ since $S=S^\top$, we split the iterates of CG into
$x=x_N+x_R$.
\\
\textit{Iterates in }null{($S$)}:
Consider  the initialization in Algorithm~\ref{alg:CG}.
As $ s\in \operatorname{range}(S)$---i.e.. $S  \lambda_N^0=s_N=0$---we have $r^{0}_N=p^{0}_N=0$.
As $\operatorname{range}(S)\supset \{0\}$, $S\succeq 0$, and  $r^n\neq 0$ due to the termination criterion,  for all iterates $n \in \mathbb{N}$, the scalars $\alpha^n,\beta^n$ are positive numbers and  depend only on  $p_R$ and $r_R$.
Hence, by \eqref{eq:CGlam}, \eqref{eq:CGr}, and \eqref{eq:CGp}, $r^{m}_N=p^{m}_N=0$ and $\lambda^{m}_N = \lambda^{0}_N$ for  $m\geq n$.
\\
\textit{Iterates in} {range($S$)}:
Since $r^{n}_N=p^{n}_N=0$, the iterates \eqref{eq:CGalpha}-\eqref{eq:CGp} are not affected by their kernel components (except for \eqref{eq:CGlam}) since only zeros are added.
Let $Q\in \mathbb{R}^{m\times m-n_0}$ be a matrix {whose} columns are an orthonormal basis of $\operatorname{range}(S)$.
Thus, one can uniquely represent all components of $ \lambda, r$ and $p$ in the range of $S$ in terms of $Q$, i.e., $( \lambda_R, r_R, p_R)=( Q\,{ \underline\lambda}_R,Q \,\underline r_R, Q \, \underline p_R)$.
Express the CG iterations \eqref{eq:CGiter} in the $\underline\cdot$ variables, left-multiplying \eqref{eq:CGlam}, \eqref{eq:CGr} and \eqref{eq:CGp} with $Q^\top $; use $Q^{-1}=Q^\top $,  $r^{n}_N=p^{n}_N=0$,  to get
\begin{align*}
	\alpha^n &= \frac{\underline r^{n\top}_R \underline r^{n}_R}{\underline  p^{n\top}_R  \underline { S }\,\underline p^{n}_R},  & \underline{ \lambda}^{n+1}_R &= \underline { \lambda}^{n}_R + \alpha^n \underline p^{n}_R, 
	\\	\underline r^{n+1}_R &= \underline r^{n}_R-\alpha^n \underline{S} \underline p^{n}_R, 
	&	\beta^n &= \frac{\underline r^{n+1 \top}_R\underline r^{n+1}_R}{\underline r^{n\top}_R \underline r^{n}_R},  \\
	\underline p^{n+1}_R &= \underline r^{n+1}_R + \beta^n \underline p^{n}_R,
\end{align*}
where $ \underline {S } \doteq Q^\top S Q $.
The same applies to the initial condition in Algorithm~\ref{alg:CG} and gives $\underline r^0_R= \underline p^0_R=\underline  {s} - \underline{S}  \, \underline { \lambda}^0$ with $\underline  {s}\doteq Q^\top s$.
By construction,  the matrices $\underline {S }$ and ${S }$ have the same non-zero eigenvalues.
Hence,  $\underline {S } \succ 0$ and thus the standard CG convergence theory  applies, which leads to the bound \eqref{eq:est2}, cf. \cite[Eq 5.36]{Nocedal2006}. 
Moreover, \cite{Greenbaum1979} has shown exact convergence after $m-n_0$ iterations  for solving $  \underline {S }\,  \underline{ \lambda }_R = \underline {s }$. 
Thus, $\underline{ \lambda}_R$ converges to the solution  $  \underline {S }\,  \underline{ \lambda }^{\star}_R = \underline {s }$.
Hence, the construction of $\underline {S }$ and $\underline {s }$ implies   $Q^\top S  \lambda_R = Q^\top s$.
Since $ \lambda ^{n}_N =  \lambda ^{0}_N$,  CG  converges to ${\lambda}^\star  = Q \underline{ \lambda}^{\star}_R +  \lambda^{0}_N$, which satisfies~\eqref{eq:SchurComp}.

\section*{Decentralized ADMM} \label{sec:ADMM}

For the sake of comparison, we 
reformulate problem~\eqref{eq:SchurComp} as the optimality conditions of a consensus problem to apply a decentralized variant of ADMM.
First we eliminate the zero rows/columns in $S_i$ to obtain a partially separable \emph{strongly} convex optimization problem using the sparsity framework from \autoref{sec:d-CG}. This ensures fast (linear) convergence for ADMM.
A lack of strong convexity might lead to a sublinear  convergence  \cite{Deng2016,Shi2014}. This is the case, if there are redundant consensus constraints---such as in \autoref{sec:App} except for the path graph.
Via  \autoref{lem:propS},  and  $AB=BA$ for diagonal matrices, the left-hand-side of \eqref{eq:SchurComp} is written as
\begin{align*} 
	&\left  (\sum_{i \in \mathcal{S}} S_i \right  ) \lambda 
	=\Bigg (\hspace{-.5mm}\sum_{i \in \mathcal{S}} I_i^{\top}  I_i^{\phantom \top}  \, S_i\, I_i^{\top}I_i^{\phantom \top} \hspace{-1mm} \Bigg ) \lambda \\
	&\qquad =\sum_{i \in \mathcal{S}}I_i^{\top}  \left ( I_i^{\phantom \top}   S_i I_i^{\top}  \right )I_i^{\phantom \top} \lambda 
	=\sum_{i \in \mathcal{S}}I_i^{\top} \, \hat S_i\, I_i^{\phantom \top} \lambda.
\end{align*}
Similarly, the right-hand-side of \eqref{eq:SchurComp} gives
$
\sum_{i \in \mathcal{S}} s_i  =\sum_{i \in \mathcal{S}} I_i^{\top} I_i^{\phantom \top} s_i   
=\hspace{-1mm}\sum_{i \in \mathcal{S}}  I_i^{\top} \hat s_i.
$
Summing up, we write \eqref{eq:SchurComp} as 
$
\sum_{i \in \mathcal{S}}I_i^{\top}     \hat    S_i  I_i \lambda
=
\sum_{i \in \mathcal{S}}  I_i^{\top}\; \hat s_i   .
$
Since $\hat S_i = \hat S_i^\top\succeq 0$, this corresponds to the optimality conditions of 
$
\min_{\lambda} \sum_{i\in \mathcal{S}} \frac{1}{2} \lambda^\top I_i^\top \; \hat S_i\;  I_i\lambda - I_i\, \hat s_i^{\top } \lambda .
$
Equivalently, we have
\begin{align} \label{eq:consLam2}
	\begin{aligned}
		\min_{\lambda_1,\dots,\lambda_S, \bar   \lambda} \quad&\sum_{i\in \mathcal{S}} \frac{1}{2} \lambda_i^\top \hat S_i \lambda_i - \hat s_i^{\top } \lambda_i \\
		\;\text{subject to } \quad& \lambda_i= I_i  \bar  \lambda\;\; | \;\; \gamma_i, \text{ for all } i \in \mathcal{S}.
	\end{aligned}
\end{align}
with new local multipliers $\lambda_1,\dots,\lambda_S$.
Decentralized ADMM follows standard patterns \cite{Bertsekas1989,Boyd2011}, see Algortihm~\ref{alg:D-ADM}, 
{where $\bar \lambda _i \doteq I_i  \bar  \lambda$.}

\begin{algorithm} [t]
	\caption{Decentralized ADMM for problem \eqref{eq:SchurComp}}
	\small
	\textbf{Initialization: $ \lambda_i^0$, $\gamma_i^0$ for all $i \in \mathcal{S}$, $\rho$}\\
	\textbf{For all $i \in \mathcal{S}$, repeat until $\|\lambda_i^n -  \bar  \lambda^n_i \| < \epsilon_p$ and  $\| \bar \lambda_i^{n+1} -  \bar \lambda_i^n  \| < \epsilon_d$:}
	\vspace{-.4cm}
	\begin{subequations}
		\begin{align}
			&\textstyle {\lambda_i^{n+1}=\left (\hat S_i + \rho I \right)^{-1}\left ( \hat s_i - \gamma_i^n + \rho \bar   \lambda^n_i \right )} & \text{(local)} \\
			&\textstyle { \bar \lambda^{n+1}_i = \left (\Lambda_i \right )^{-1} \sum_{j \in \mathcal  N_i} I_{ij} \lambda_j^{n+1} } 
			&\hspace{-.6cm} \text{(neighbor-neighbor)} \label{step:ADMavg} \\
			&\gamma_i^{n+1} = \gamma_i^n + \rho(\lambda^{n+1}_i -   \bar  \lambda_i^{n+1}) & \text{(local)} \\
			& n \leftarrow n+1 & \notag
		\end{align}
	\end{subequations} \label{alg:D-ADM} 
	\vspace{-.4cm}
\end{algorithm}

	\renewcommand*{\bibfont}{\footnotesize}
	{\footnotesize
	\printbibliography}
	
\end{document}